 \documentclass{amsart}

\usepackage{pdfsync}

 \usepackage[plainpages=false]{hyperref}
 \usepackage{amsfonts,amsmath,amssymb,amsthm,accents}
 \usepackage{latexsym,lscape,rawfonts,mathrsfs}

 \usepackage[all]{xy}
 \usepackage{eufrak}
 \usepackage{graphicx,psfrag}
 \usepackage{pstool}

 \usepackage{array,tabularx}

 \usepackage{appendix}

\usepackage{txfonts}

 \newcommand{\E}{\ensuremath{\mathbb{E}}}

 \newcommand{\R}{\ensuremath{\mathbb{R}}}

 \newcommand{\ba}{\begin{align}}
 \newcommand{\ea}{\end{align}}
 \newcommand{\bal}{\begin{align*}}
 \newcommand{\eal}{\end{align*}}

 \DeclareMathOperator{\supp}{supp}

 \newcommand{\Rm}{\mathcal{R}m}
 \newcommand{\Rc}{\mathcal{R}c}

 \newcommand{\dvol}{\text{d}V}
\newcommand{\dmu}{\text{d}\mu}

\renewcommand{\epsilon}{\varepsilon}

 \makeatletter
 \def\ExtendSymbol#1#2#3#4#5{\ext@arrow 0099{\arrowfill@#1#2#3}{#4}{#5}}
 
 \makeatother

 \makeatletter
 \def\ExtendSymbol#1#2#3#4#5{\ext@arrow 0099{\arrowfill@#1#2#3}{#4}{#5}}
 
 \makeatother

 \definecolor{hao}{rgb}{1,0.5,0}
 \definecolor{miao}{cmyk}{0.5,0,0.2,0.2}
 \definecolor{qiao}{gray}{0.96}

\newtheorem{prop}{Proposition}[section]

\newtheorem{proposition}[prop]{Proposition}

\newtheorem{theorem}[prop]{Theorem}

\newtheorem{lemma}[prop]{Lemma}

\newtheorem{corollary}[prop]{Corollary}

\newtheorem*{theorem*}{Theorem}

\numberwithin{equation}{section}

\keywords{Harmonic map, Heat kernel, Ricci curvature}

\title{Rigidity of vector valued harmonic maps of linear growth}
\author[Shaosai Huang and Bing Wang]{Shaosai Huang\quad and\quad Bing Wang}

\address{Shaosai Huang, Department of Mathematics, Stony Brook University, 100
Nicolls Road, Stony Brook, NY, 11794-3651, U.S.A.}
\email{shaosai.huang@stonybrook.edu}

\address{Bing Wang, Department of Mathematics, University of Wisconsin -
Madison, 480 Lincoln Drive, Madison, WI, 53706, U.S.A.}
\email{bwang@math.wisc.edu}

\date{\today}
\thanks{Both authors are partially supported by NSF grant DMS-1510401.}

\begin{document}
\maketitle
\begin{abstract}
Consider vector valued harmonic maps of at most linear growth, defined on a
complete non-compact Riemannian manifold with non-negative Ricci curvature. For
the norm square of the pull-back of the target volume form by such maps, we
report a strong maximum principle, and equalities among its supremum, its asymptotic
average, and its large-time heat evolution.
\end{abstract}


\section{Introduction}
We fix a complete, non-compact Riemannian manifold $(M,g)$ with
non-negative Ricci curvature. Harmonic functions on such manifolds have been an
important subject of investigation in geometric analysis (see, among others,
\cite{Yau75}, \cite{Li84},\cite{Li86}, \cite{LiSc},
\cite{DF},\cite{LiT1},\cite{LiT2},\cite{CoMi1}, \cite{CoMi2}, and \cite{CCM},
etc.).

Among all harmonic functions, those of at most linear growth are especially
interesting, as they reveal the properties of the tangent cones at infinity of
$M$, through the work of Cheeger-Colding-Minicozzi II~\cite{CCM}. Vector valued
harmonic maps of at most linear growth are also of central imporance in studying
the structure of the Gromov-Hausdorff limits of a sequence of Riemannian
manifolds with Ricci curvature uniformly bounded below, as signified in the
series of work by Cheeger-Colding (see \cite{ChCo}, \cite{ChCo1}, \cite{ChCo2},
and \cite{ChCo3}), and notably the recent resolution of the codimension-4
conjecture by Cheeger-Naber~\cite{ChNa14}.

Our subject of study is a vector valued harmonic map of at most
linear growth. To fix the terminologies, we call a function $\mathbf{u}:
(M^m,x_0) \rightarrow (\R^n,\mathbf{0})$ ($n\le m$) a \emph{vector valued
harmonic map} if each component of $\mathbf{u}$ is a harmonic function
$u_{\alpha}:M\rightarrow \R$ ($\alpha=1,\cdots,n$), and call it \emph{pointed}
if $\mathbf{u}(x_0)=\mathbf{0}$. Moreover, we say that $\mathbf{u}$ is \emph{of
at most linear growth} if there exists some $L>0$, such that $$\forall x\in
M,\quad |\mathbf{u}(x)-\mathbf{u}(x_0)|\ \leq\ L\ (r(x)+1),$$ where $\forall
x\in M,\ r(x):=d(x,x_0)$, the geodesic distance between $x$ and $x_0$, induced
by the Riemannian metric $g$.

To each vector valued harmonic map as above, we could associate an $n$-form
$$ \omega\ :=\mathbf{u}^*(dy^1 \wedge dy^2 \wedge \cdots \wedge dy^n)\ \in\
\Gamma(M,\Lambda^nT^{\ast}M).$$ We say that $\mathbf{u}$ is \emph{non-trivial}
if $|\omega|\not\equiv 0$.

For any $A\in GL(n)$, define $\omega_A:=(A\mathbf{u})^{\ast}(dy^1\wedge
dy^2\wedge\cdots\wedge dy^n)$. The following invariance of the pull-back
$n$-form under the induced $SL(n)$-actions is fundamental for our arguments:
\begin{align}
\forall A\in SL(n),\quad |\omega_A|^2\ =\ \det A\ |\omega|^2\ =\ |\omega|^2.
\label{eqn:SLk_invariance}
\end{align}

Our first result then states:
\begin{theorem}[\textbf{Strong maximum principle}]
Let $\mathbf{u}: (M,x_0)\rightarrow (\R^n,\mathbf{0})$ be a pointed, vector
valued harmonic map of at most linear growth which is non-trivial at $x_0$.
 Then 
\begin{align}
\lim_{\rho\rightarrow \infty}\fint_{B(x_0,\rho)}|\omega|^2\ \dvol_g\ =\
\sup_M|\omega|^2.
\label{eqn:max_principle}
\end{align}
Moreover, if $|\omega|(x_0)=\sup_M|\omega|$, then $|\nabla \nabla \mathbf{u}|
\equiv 0$, and $(M^m, g)\equiv (N^{m-n},h) \times (\R^n,g_{Euc})$
isometrically, with $(N^{m-n},h)$ being some $(m-n)$-dimensional Riemannian
manifold with non-negative Ricci curvature.
\label{thm:main}
\end{theorem}
\noindent Clearly, there is nothing special about the choice of base point
$x_0$ in the statement of the theorem.
Our proof of this theorem is a blend of the heat kernel estimates due to
Li-Yau (see~\cite{LiYau} and~\cite{Li84}), and the Hessian $L^2$-estimates by
Cheeger-Colding-Minicozzi II~\cite{CCM}. Now we briefly discuss the ingredients
involved in Theorem~\ref{thm:main}.

Notice, since $\Rc \ge 0$, that each individual $|\nabla u_{\alpha}|^2$
($\alpha=1,\cdots,n$) is sub-harmonic by the Weitzenb\"ock formula:
\begin{align}
\Delta |\nabla u_{\alpha}|^2\ =\ 2|\nabla\nabla u_{\alpha}|^2+2\Rc(\nabla
u_{\alpha},\nabla u_{\alpha})\ \ge\ 0,
\label{eqn:Bochner_intro}
\end{align}
also notice that the linear growth of $u_{\alpha}$ gives a global upper bound
of $|\nabla u_{\alpha}|^2$ by the Cheng-Yau gradient estimate~\cite{ChengYau}:
$|\nabla u_{\alpha}|^2$ is a bounded, non-negative sub-harmonic function on
$M$. By a classical theorem of Peter Li~\cite{Li86}, we have
\begin{align}
\lim_{\rho\to \infty}\fint_{B(x_0,\rho)}|\nabla u_{\alpha}|^2\ \dvol_g\ =\
\sup_M|\nabla u_{\alpha}|^2.
\label{eqn:Li1}
\end{align}
Therefore we could view (\ref{eqn:max_principle}) as a high-dimensional
generalization of this identity for the energy of harmonic functions of at most
linear growth.

The proof of (\ref{eqn:max_principle}) is based on the invariance of certain
canonical quantities associated to a vector valued harmonic map $\mathbf{u}$,
and the invariance properties enable the compactness of $SO(n)$ to work for
limiting arguments.
Besides the pull-back measure density of $\mathbf{u}$, we also define: 
\begin{enumerate}
\item Energy density:\quad $|\nabla \mathbf{u}|^2\ :=\ \sum_{\alpha=1}^n|\nabla
u_{\alpha}|^2$;
\item Splitting error:\quad $r^2|\nabla\nabla \mathbf{u}|^2\ :=\
\sum_{\alpha=1}^nr^2|\nabla\nabla u_{\alpha}|^2$.
\end{enumerate}
 We notice that the quantities $|\omega|^2$, $|\nabla \mathbf{u}|^2$ and
 $r^2|\nabla\nabla \mathbf{u}|^2$ are invariant under the special orthogonal
 group actions on $\R^n$:
\begin{align}
\forall A\in SO(n),\quad \mathbf{u}\ \mapsto\ A\mathbf{u},\quad i.e.\ 
(A\mathbf{u})_{\alpha}\ :=\ \sum_{\beta=1}^nA_{\alpha\beta}u_{\beta}.
\label{eqn:SOk_invariance}
\end{align}

For the rigidity part of Theorem~\ref{thm:main}, i.e. when
$|\omega|(x_0)=\sup_M|\omega|$ happens, we observe that
(\ref{eqn:Bochner_intro}) links the difference of the average energy between
two scales, and this observation is implemented through the application of the
heat measure. Here the \emph{heat measure} is defined as
$\dmu_{x_0}(t):=H_{x_0}(\cdot,t)\ \dvol_g$, where $H_{x_0}(y,t)$ is the
fundamental solution to the heat equation (in $y\in M$) with the Delta function
at $x_0\in M$ as its initial value.
Summing (\ref{eqn:Bochner_intro}) in $\alpha=1,\cdots,n$, the aforementioned
observation could be expressed as (see~\cite{Li84} for the justification of
integration by parts)
\begin{align*}
\frac{\text{d}}{\text{d}t}\int_M|\nabla \mathbf{u}|^2\ \dmu_{x_0}(t)\ \ge\
2\int_M|\nabla\nabla \mathbf{u}|^2\ \dmu_{x_0}(t),
\end{align*}
so integrating between different scales $\rho_1<\rho_2$, we get
\begin{align}
\int_M|\nabla \mathbf{u}|^2\ \dmu_{x_0}(\rho_2^2)-\int_M|\nabla\mathbf{u}|^2\
\dmu_{x_0}(\rho_1^2)\ \ge\
2\int_{\rho_1^2}^{\rho_2^2}\int_M|\nabla\nabla\mathbf{u}|^2\
\dmu_{x_0}(t)\text{d}t.
\label{eqn:splitting_error_intro}
\end{align}
This inequality contains rich information about $\mathbf{u}$: it tells not only
that the weighted energy (weighted by the heat measure based at $x_0\in M$) on
large scales dominates that on smaller ones, but also that the difference
dominates the splitting error. Roughly speaking, if $|\omega|^2$ attains a
global maximum at $x_0$, this inequality then forces the splitting error to
vanish.

Following a similar argument and with the help of a Poincar\'e inequality
weighted by the heat measure (Lemma~\ref{lem:Weight_Poincare}), we also
prove:
\begin{theorem}[\textbf{Large-time heat evolution}]
Let $(M,g,x_0)$ and $\mathbf{u}$ satisfy the assumption of
Theorem~\ref{thm:main}, then
\begin{align}
\lim_{t\to \infty}\int_M|\omega|^2\ \dmu_{x_0}(t)\ =\ \sup_M|\omega|^2.
\end{align}
\label{thm:second}
\end{theorem}
This theorem, in conjunction with Theorem~\ref{thm:main}, then says
\begin{align}
\lim_{\rho\to \infty}\fint_{B(x_0,\rho)}|\omega|^2\ \dvol_g\ =\
\sup_M|\omega|^2\ =\ \lim_{t\to \infty}\int_M|\omega|^2\ \dmu_{x_0}(t).
\label{eqn:identity}
\end{align} 

Notice that $|\omega|^2$ is \emph{not} necessarily sub-harmonic in any obvious
way (see Section 4 for more details), therefore we have to prove both sides of
(\ref{eqn:identity}) separately.
At this point, it is interesting to compare this identity with classical results
of Peter Li in~\cite{Li86}:
\begin{theorem}[Peter Li~\cite{Li86}]
 Suppose $M$ has positive asymptotic volume ratio at infinity, i.e. 
$\lim_{\rho\to \infty}\omega_m^{-1}\rho^{-m}Vol_g(B(x_0,\rho))\ =\ \kappa\ >0,$
then for any bounded function $f$,
\begin{align}
\ \lim_{\rho\to\infty}\fint_{B(x_0,\rho)}f\ \dvol_g\ =\ \lim_{t\to
\infty}\int_Mf\ \dmu_{x_0}(t),
\label{eqn:Li2}
\end{align}
as long as one of these limits exists.

On the other hand, if $f$ is a bounded, non-negative sub-harmonic function,
then regardless of the positivity of the asymptotic volume ratio at infinity,
\begin{align}
\lim_{\rho\to \infty}\fint_{B(x_0,\rho)}f\ \dvol_g\ =\ \sup_M f\quad\text{and}\quad
\lim_{t\to \infty}\int_M f\ \dmu_{x_0}(t)\ =\ \sup_M f.
\label{eqn:Li3}
\end{align}

\label{thm:Li1}
\end{theorem}

Notice that when $M$ has vanishing asymptotic volume ratio at infinity,
(\ref{eqn:Li2}) is not necessarily true for \emph{any} bounded function,
see~\cite{Xu14}. However it is not clear if (\ref{eqn:Li3}) \emph{only} holds
for sub-harmonic functions. Therefore, (\ref{eqn:identity}) provides another
incidence when (\ref{eqn:Li3}) stands for a function which is \emph{not}
necessarily sub-harmonic, \emph{regardless} of the positivity of the asymptotic
volume ratio at infinity. See Section 4 for a further discussion.

\section{Background}
In this section we review the relevant facts needed for our future arguments:
Cheeger-Colding's segment inequality, and the Li-Yau heat kernel estimates.
We will also recall a result due Peter Li, the proof of which was embedded in
other results and here we single it out. Instead of citing the original
results in full generality, we will state the results in a form that suit our
future applications.

Given any Riemannian manifold with Ricci curvature uniformly bounded below, a
fundamental inequality which is directly built on the Bishop-Gromov volume
comparison is Cheeger-Colding's segment inequality~\cite{ChCo} (see
also~\cite{Cheeger_Millan}): 
\begin{proposition}[\textbf{Segment inequality}]
Let $(M^m,g)$ be a complete Riemannian manifold with $\Rc_g\ge 0g$. Let
$B(x_0,r)$ be a geodesic ball of radius $\rho>0$ around $x_0\in M$. For any
$f\in L^1_{loc}(M)$ we define
\begin{align*}
\mathcal{F}(x,y)\ :=\ \inf_{\gamma_{xy}}\int_{0}^{d(x,y)}f(\gamma_{xy}(t))\
\text{d}t,
\end{align*}
where the infimum is taken over all minimal geodesics $\gamma_{xy}$ connecting
$x$ and $y$. There is a dimensional constant $C_{CC}(m)>0$ such that for any
$f\in L^1_{loc}(M)$,
\begin{align*}
\int_{B(x_0,\rho)\times B(x_0,\rho)}\mathcal{F}(x,y)\ \dvol_g(x)\dvol_g(y)\ \le\
C_{CC}(m)|B(x_0,\rho)|(2\rho)\ \int_{B(x_0,2\rho)}f\ \dvol_g.
\end{align*} 
\label{prop:segment}
\end{proposition}
This inequality is useful in extracting estimates along \emph{most} geodesics
connecting pairs of points, see for instance~\cite{ChCo} and~\cite{CoNa11}. It
also has the Poincar\'e inequality as a natural consequence (see~\cite{SC92}
and~\cite{Cheeger_Millan}):
\begin{proposition}[\textbf{Poincar\'e inequality}]
Assume $(M,g)$ is a complete Riemannian manifold with non-negative Ricci
curvature. There is a dimensional constant $C_P(m)>0$ such that for any $f\in
W^{1,1}_{loc}(M)$,
\begin{align}
\fint_{B(x_0,\rho)}\left|f-\fint_{B(x_0,\rho)}f\ \dvol_g\right|\ \dvol_g\ \le\
C_P(m)\rho \fint_{B(x_0,2\rho)}|\nabla f|\ \dvol_g.
\label{eqn:Poincare}
\end{align}
\label{prop:Poincare}
\end{proposition}

In fact, with the help of the segment inequality, we will later prove a version
of the Poincar\'e inequality (Lemma~\ref{lem:Weight_Poincare}) with the heat
measure $\dmu_{x_0}$ replacing the volume form $\dvol_g$. To set up the weighted
Poincar\'e inequality, the heat kernel estimates due to
Li-Yau~\cite{LiYau} is also of crucial importance:
\begin{proposition}[\textbf{Heat kernel bounds}]
Let $(M,g)$ be a complete, non-compact Riemannian manifold with non-negative
Ricci curvature. Then the fundamental solution $H_{x_0}(x,t)$ to the heat
equation satisfies
\begin{align}
\frac{C_1(\varepsilon)^{-1}}{|B(x_0,\sqrt{t})|}
e^{-\frac{r^2(x)}{(4-\varepsilon)t}}\  \le\ H_{x_0}(x,t)\ \le\
\frac{C_2(\varepsilon)}{|B(x_0,\sqrt{t})|}e^{-\frac{r^2(x)}{(4+\varepsilon)t}},
\label{eqn:kernel_estimate}
\end{align}
where the positive constants $C_1(\varepsilon),C_2(\varepsilon)\to \infty$ as
$\varepsilon\to 0$.
\label{prop:kernel_estimate}
\end{proposition}
Moreover, we will need the following Harnack inequality for heat kernels,
which is also proved in~\cite{LiYau}:
\begin{proposition}[\textbf{Harnack inequality for positive heat equation
solutions}] Let $w(x,t)$ be a positive solution to the heat equation, then for
$x,y$ and $t_1<t_2$ we have
\begin{align}
w(x,t_1)\ \le\
w(y,t_2)\left(
\frac{t_2}{t_1}\right)^{\frac{3m}{4}}\exp\left(\frac{3d(x,y)^2}{8(t_2-t_1)}\right).
\label{eqn:Harnack1}
\end{align}
\label{prop:Harnack}
\end{proposition}

As a consequence of the above Harnack inequality, we have the following lemma
due to Peter Li~\cite{Li86}, see the second identity of (\ref{eqn:Li3}):
\begin{lemma}
Let $u$ be a bounded, non-negative sub-harmonic function defined on a complete
non-compact Riemannian manifold $(M,g)$ with non-negative Ricci curvature. Let
$\dmu_{x}$ be the heat measure based at $x\in M$, then
\begin{align}
\lim_{t\to \infty}\int_Mu\ \dmu_{x}(t)\ =\ \sup_Mu.
\label{eqn:heat_max}
\end{align}
\label{lem:heat_max}
\end{lemma}
\noindent The proof of this lemma is contained in the proof of Proposition 2
in~\cite{Li86}. We include it here for the sake of completeness.
\begin{proof}
For any $x\in M$ and $t>0$ we define $Hu(x,t):=\int_Mu\ \dmu_x(t)$. By the
sub-harmonicity of $u$ we see that $Hu(x,t)$ is non-decreasing in $t$:
\begin{align*}
\forall x\in M,\ \forall t>0,\quad \partial_tHu(x,t)\ =\ \int_M\Delta u\
\dmu_x(t)\ \ge\ 0. 
\end{align*}
Especially $Hu(x,t)\ge u(x)$ for any $x\in M$ and any $t>0$. 
Moreover, the positivity of $H_x(y,t)$ (the fundamental solution to the heat
equation based at $x\in M$) and the stochastic completeness of $\dmu_x(t)$
ensure that 
\begin{align*}
\forall x\in M,\ \forall t>0,\quad 0\ \le\ Hu(x,t)\ \le\ \sup_Mu.
\end{align*}

With the help of (\ref{eqn:Harnack1}), we see that as $t\to \infty$, $Hu(-,t)$
converges to a function uniformly on compact subsets of $M$. Choosing a compact
exhaustion of $M$ and using a diagonal argument of choosing sub-sequences, we
conclude that as $t\to \infty$, $Hu(-,t)$ converges to some globally defined
function $Hu(-,\infty)$, uniformly on compact subsets of $M$.

Clearly $0\le Hu(-,\infty) \le \sup_Mu$, and in fact $Hu(-,\infty)$ is harmonic. 
Therefore, by the Cheng-Yau gradient estimate~\cite{ChengYau}, $Hu(-,\infty)$
must be a constant, which must be $\sup_Mu$, since $H(x,\infty)\ge u(x)$ for
any $x\in M$.
\end{proof}



\section{The strong maximum principle}
In this section we prove the strong maximum principle, Theorem~\ref{thm:main}.
Consider the functions $E_{\alpha\beta}(\mathbf{u}):=\langle \nabla
u_{\alpha},\nabla u_{\beta}\rangle$ on $M$ ($\alpha,\beta=1,\cdots,n$), and the
$n\times n$-matrix valued function
$\E(\mathbf{u}):=[E_{\alpha\beta}(\mathbf{u})]$, which is positive
semi-definite throughout $M$.
For each $\rho>0$, we could also consider the average of $\E(\mathbf{u})$ over
$B(x_0,\rho)$, a positive semi-definite numerical matrix:
\begin{align*}
\Omega_{\rho}(\mathbf{u})\ :=\
\left[\fint_{B(x_0,\rho)}E_{\alpha\beta}(\mathbf{u})\
\dvol_g\right],
\end{align*}
 
 Since $|\omega|^2=\det \E(\mathbf{u})$, the theme of the proof is the interplay
 between the quantities $\fint_{B(x_0,\rho)}\det\E(\mathbf{u})\ \dvol_g$ and
 $\det \Omega_{\rho}(\mathbf{u})$, which are obtained by taking determinant and
 then average of $\E(\mathbf{u})$, or in the alternative order.

 Pick a sequence of scales $\rho_i$ that tend to infinity. By Li's identity
 (\ref{eqn:Li1}), the adjusted harmonic maps (adjustments made so that
 $\Omega_{\rho_i}$ becomes diagonal) will see $\det \Omega_{\rho_i}$
 approaching $\sup_M|\omega|^2$, since $\det \Omega_{\rho_i}$ and $|\omega|^2$
 are invariant under the induced $SO(n)$-actions, and the diagonalization
 enables us to deal with the determinants in a similar way as scalar functions.
 
  On the other hand, following an argument in~\cite{CCM}, the $L^2$-average of 
  Hessian could be shown to approach zero, therefore, by the Poincar\'e
  inequality (\ref{eqn:Poincare}), the process of taking determinant and taking
  average of $\E(\mathbf{u})$ on larger and larger scales will gradually commute.
 
 The last piece is the compactness of $SO(n)$, from which we could obtain
 certain limiting adjustment by a special orthogonal matrix that works for the
 asymptotic behavior.

The proof of the rigidity part follows from a direct application of the heat
measure, as outlined in the introduction. We also need Lemma~\ref{lem:heat_max}
to take care of individual component functions.
\begin{proof}[Proof of Theorem~\ref{thm:main}]
\textbf{\textit{Asymptotic maximality.}} To prove (\ref{eqn:max_principle}),
notice that it suffices to show the following:
\begin{align}
\forall x\in M,\quad|\omega|^2(x)\ \le\ \lim_{\rho\to
\infty}\fint_{B(x_0,\rho)}|\omega|^2\ \dvol_g,
\label{eqn:asymptotic_omega}
\end{align}
since $\fint_{B(x_0,\rho)}|\omega|^2\ \dvol_g\le \sup_M|\omega|^2$ for any $\rho>0$.
To prove (\ref{eqn:asymptotic_omega}), pick any sequence $\rho_i\to \infty$ and
let $A_i\in SO(n)$ diagonalize $\Omega_{\rho_i}(\mathbf{u})$. Then by the
compactness of $SO(n)$, possibly passing to a subsequence, $A_i\to
A_{\infty}\in SO(n)$. Denoting $\mathbf{v}:=A_{\infty}\mathbf{u}$, we have,
according to (\ref{eqn:SLk_invariance}) and (\ref{eqn:SOk_invariance}),
\begin{align}
\label{eqn:det_invariance}
\begin{split}
&\det \Omega_{\rho_i}(\mathbf{u})\ =\ \det\Omega_{\rho_i}(A_i\mathbf{u})\ =\
\det\Omega_{\rho_i}(\mathbf{v}),\\
 \text{and}\quad &|\omega|^2\ \equiv\ |\omega_{A_i}|^2\ \equiv \
 |\omega_{A_{\infty}}|^2\ \text{on}\ M.
\end{split}
\end{align}
Moreover, for $\lambda_{\alpha,i}^2:=\fint_{B(x_0,\rho_i)}|\nabla v_{\alpha}|^2\ \dvol_g$, we have, by the convergence $A_i\to A_{\infty}$, that
\begin{align}
\lim_{i\to\infty}\left|\prod_{\alpha=1}^n\lambda^2_{\alpha,i}-\det\Omega_{\rho_i}(\mathbf{u})\right|\ =\ 0.
\label{eqn:det_limit}
\end{align}
Since $A_{\infty}$ is a linear transformation, each component function of $\mathbf{v}=A_{\infty}\mathbf{u}$ is harmonic, and applying (\ref{eqn:Li1}) we have, for each $\alpha=1,\cdots,n$,
\begin{align}
\lim_{i\to \infty}\lambda_{\alpha,i}^2\ =\ \lim_{\rho\to\infty}\fint_{B(x_0,\rho)}|\nabla v_{\alpha}|^2\ \dvol_g\ =\ \sup_M|\nabla v_{\alpha}|^2\ =:\ L_{\alpha}^2.
\label{eqn:Li_v_i}
\end{align}
Now we follow an argument in~\cite{CCM} to control the average Hessian of each $v_{\alpha}$ on large enough scales.
For any fixed $\rho_i>0 $, let $\varphi_i$ be a cutoff function defined as in~\cite{SY94}, such that $\supp \varphi_i\subset B(x_0,3\rho_i)$ and $\varphi_i\equiv 1$ on $B(x_0,2\rho_i)$, moreover, 
$$\rho_i^2|\Delta \varphi_i|+\rho_i|\nabla \varphi_i|\ \le\ C(M).$$ 
We can then estimate the average Hessian on scale $\rho_i$:
\begin{align*}
\fint_{B(x_0,2\rho_i)}2|\nabla\nabla v_{\alpha}|^2\ \dvol_g\quad 
 \le\quad &2^m\fint_{B(x_0,3\rho_i)}\varphi_i \Delta\left(|\nabla
 v_{\alpha}|^2-L_{\alpha}^2\right)\ \dvol_g\\
\le\quad &2^m\fint_{B(x_0,3\rho_i)}|\Delta \varphi_i|\left(L_{\alpha}^2-|\nabla
v_{\alpha}|^2\right)\ \dvol_g\\
\le\quad &C(M)\ \rho_i^{-2} L_{\alpha}^2\ \Psi(\rho_i^{-1}),
\end{align*}
where by (\ref{eqn:Li_v_i}), $\Psi(\rho_i^{-1})>0$ satisfies
$$\Psi(\rho_i^{-1}):=\max_{\alpha=1,\cdots,n}\left(1-L_{\alpha}^{-2}\fint_{B(x_0,3\rho_i)}|\nabla v_{\alpha}|^2\ \dvol_g\right)\ \to\ 0\quad\text{as}\quad i\to\infty.$$
Consequently, we have for each $\alpha=1,\cdots,n,$
\begin{align}
\quad \rho_i^2\fint_{B(x_0,2\rho_i)}|\nabla \nabla u_{\alpha}|^2\ \dvol_g\ \le\ C(M) L_{\alpha}^2\ \Psi(\rho_i^{-1}).
\label{eqn:Hessian_control}
\end{align}
This estimate, together with the Poincar\'e inequality (\ref{eqn:Poincare}),
controls the behavior of the average pull-back measure density on large scales:
\begin{align*}
&\fint_{B(x_0,\rho_i)}\left||\omega_{A_{\infty}}|^2-\det \Omega_{\rho_i}(\mathbf{v})\right|\ \dvol_g\\
\le\quad &\sum_{\sigma\in S_n}\fint_{B(x_0,\rho_i)}\left|\prod_{\alpha=1}^n\langle \nabla v_{\alpha},\nabla v_{\sigma(\alpha)}\rangle-\prod_{\alpha=1}^n\fint_{B(x_0,\rho_i)}\langle \nabla v_{\alpha},\nabla v_{\sigma(\alpha)}\rangle\ \dvol_g\right|\ \dvol_g\\
\le\quad& \sum_{\sigma\in S_n}\sum_{\alpha=1}^n\left(\prod_{\beta\not=\alpha}L_{\beta}L_{\sigma(\beta)}\right)\fint_{B(x_0,\rho_i)}\left|\langle \nabla v_{\alpha},\nabla v_{\sigma(\alpha)}\rangle-\fint_{B(x_0,\rho_i)}\langle \nabla v_{\alpha},\nabla v_{\sigma(\alpha)}\rangle\ \dvol_g\right| \dvol_g,
\end{align*}
now for each $\alpha=1,\cdots,n$ and $\sigma \in S_n$ (the $n$-symmetric
group), we apply the Poincar\'e inequality (\ref{eqn:Poincare}) and the H\"older
inequality and (\ref{eqn:Hessian_control}) to see:
\begin{align*}
&\fint_{B(x_0,\rho_i)}\left|\langle \nabla v_{\alpha},\nabla
v_{\sigma(\alpha)}\rangle-\fint_{B(x_0,\rho_i)}\langle \nabla v_{\alpha},\nabla
v_{\sigma(\alpha)}\rangle\ \dvol_g\right| \dvol_g\\
 \le\quad &C_{P}\ \rho\fint_{B(x_0,2\rho_i)}|\nabla \langle \nabla
 v_{\alpha},\nabla v_{\sigma(\alpha)}\rangle|\ \dvol_g\\
\le\quad& 2C_P(m)C(M)\  L_{\alpha}L_{\sigma(\alpha)}
\Psi(\rho_i^{-1})^{\frac{1}{2}}.
\end{align*}

Consequently,
\begin{align}
&\fint_{B(x_0,\rho_i)}\left||\omega_{A_{\infty}}|^2-\det
\Omega_{\rho_i}(\mathbf{v})\right|\ \dvol_g\ \le\ 2C_P(m)C(M)n!n\
\left(\prod_{\alpha=1}^nL_{\alpha}^2\right)\ \Psi(\rho_i^{-1})^{\frac{1}{2}}.
\label{eqn:det_Omega_i}
\end{align}

Combining (\ref{eqn:det_invariance}), (\ref{eqn:det_limit}), (\ref{eqn:Li_v_i})
and (\ref{eqn:det_Omega_i}), we see that
\begin{align}
\forall x\in M,\quad \lim_{i\to \infty}\fint_{B(x_0,\rho_i)}|\omega|^2\
\dvol_g\ =\ \prod_{\alpha=1}^nL_{\alpha}^2\ \ge\ |\omega|^2(x),
\label{eqn:asymptotic_inequality}
\end{align}
where the last inequality stands as $\forall x\in M$,
\begin{align*}
|\omega|^2(x)\ =\ |\omega_{A_{\infty}}|^2(x)\ =\ |\nabla v_1\wedge\cdots\wedge
\nabla v_n|^2(x)\ \le\ \prod_{\alpha=1}^n|\nabla v_{\alpha}|^2(x).
\end{align*}

Since for every sequence $\rho_i\to \infty$, there is a subsequence for which 
(\ref{eqn:asymptotic_inequality}) holds, we have finished proving
(\ref{eqn:asymptotic_omega}).\\


\noindent \textbf{\textit{Rigidity.}} When $|\omega|(x_0)=\sup_M|\omega|$, we
may renormalize $\tilde{v}_{\alpha}:=L_{\alpha}^{-1}v_{\alpha}$, so that
$|\nabla \tilde{v}_{\alpha}|^2\le 1$ and $\sup_M|\nabla \tilde{v}_{\alpha}|=1$
for $\alpha=1,\cdots, n.$ By (\ref{eqn:det_limit}) and (\ref{eqn:Li_v_i}), we
see $\det \Omega_{\rho_i}(\tilde{\mathbf{v}})\to 1$. This fact, together with
(\ref{eqn:det_Omega_i}) and (\ref{eqn:asymptotic_inequality}), ensure
$$|\tilde{\omega}|(x_0)\ =\ \sup_M|\tilde{\omega}|\ =\ 1.$$ On the other hand,
since each $|\nabla \tilde{v}_{\alpha}|^2$ is bounded, non-negative and
sub-harmonic, by Lemma~\ref{lem:heat_max} we have
$$
\lim_{\rho\to\infty}\int_M|\nabla \tilde{\mathbf{v}}|^2\ \dmu_{x_0}(\rho^2)\ =\
n.$$ By (\ref{eqn:splitting_error_intro}), the splitting error between the zero
scale and the infinity scale is controlled as
 \begin{align}
 \label{eqn:dominance}
  \begin{split}
    |\nabla \tilde{\mathbf{v}}|^2(x_0)\quad 
     \le\quad &|\nabla \tilde{\mathbf{v}}|^2(x_0)+\int_0^{\infty} \int_{M}
     |\nabla \nabla \tilde{\mathbf{v}}|^2\ \dmu_{x_0}(t)\text{d}t\\
      \le\quad &\lim_{\rho \to \infty} \int_{M} |\nabla \tilde{\mathbf{v}}|^2\
      \dmu_{x_0}(\rho^2)\\
       =\quad &n.
 \end{split}
 \end{align}
However, applying the arithmetic mean - geometric mean inequality at $x_0$, we
have $$n\ =\ n|\tilde{\omega}|^{\frac{2}{n}}(x_0)\ \le\ |\nabla
\tilde{\mathbf{v}}|^2(x_0),$$ which, together with (\ref{eqn:dominance}),
ensures all inequalities there to be equalities; especially
\begin{align}
\int_0^{\infty}\int_M|\nabla\nabla\tilde{\mathbf{v}}|^2\
\dmu_{x_0}(t)\text{d}t\ =\ 0\quad \Longrightarrow\quad \forall t>0,\quad
\int_M|\nabla\nabla \tilde{\mathbf{v}}|^2\ \dmu_{x_0}(t)\ =\ 0,
\end{align} 
whence the vanishing of $|\nabla\nabla \tilde{\mathbf{v}}|^2$ throughout $M$.

Since $\tilde{\mathbf{v}}$ is obtained from $\mathbf{u}$ by a linear
transformation, this shows that $|\nabla\nabla \mathbf{u}|^2\equiv 0$ on $M$.
The isometric splitting then follows easily from the vanishing of the splitting
error of $\mathbf{u}$ and its non-degeneracy. For the details of the argument,
see for instance~\cite{CG71}.
\end{proof}

  \begin{corollary}
  Suppose $(M,g)$ and $\mathbf{u}:M\to \R^n$ satisfy the same assumptions as in
  the strong maximum principle.
  We have $|\nabla \nabla \mathbf{u}| \equiv 0$ whenever one of the following
  conditions are satisfied:
  \begin{enumerate}
  \item $|\nabla \omega|\ \equiv\ 0$;
  \item $\Delta |\omega|^2\ \equiv\ 0$;
  \item $\Delta |\nabla \omega|^2\ \equiv\  0$
  \end{enumerate}
  \label{thm:ME02_1}
\end{corollary}
\begin{proof}
Condition (1) implies that $|\omega| \equiv constant$ by a direct computation. 

For Condition (2), first notice that $|\omega|^2=|\nabla u_1 \wedge \nabla u_2
\wedge \cdots \wedge \nabla u_n|^2$ is uniformly bounded throughout $M$,
according to the control of $\mathbf{u}$ and Cheng-Yau gradient
estimate~\cite{ChengYau} applied to each component of $\mathbf{u}$. Since
Condition (2) says that $|\omega|^2$ is actually harmonic, applying
Cheng-Yau's gradient estimate again to $|\omega|^2$ we immediately see that
$|\omega|^2$ is a constant.

Since on $M$ any non-negative harmonic function is constant (Corollary 1
of~\cite{Yau75}), Condition (3) implies that $|\nabla \omega|^2\equiv C\ge 0$.
However, $|\nabla \omega|^2\le C(m,n)L^{n-2}|\nabla \nabla \mathbf{u}|^2$ on
$M$ and according to (\ref{eqn:Hessian_control}),
$$\lim_{\rho\to \infty}\fint_{B(x,\rho)}|\nabla \nabla \mathbf{u}|^2\ \dvol
=0.$$
Therefore, we have
\begin{align*}
 \lim_{\rho\to\infty}\fint_{B(x,\rho)}|\nabla \omega|^2\ \dvol\quad
 \le\quad &C(m,n)L^{n-2}\lim_{\rho\to \infty}\fint_{B(x,\rho)}|\nabla \nabla
\mathbf{u}|^2\ \dvol\\
 =\quad &0,
\end{align*} 
whence the constancy of $|\omega|$ throughout $M$.
\end{proof}

\section{Weighted Poincar\'e inequality and applications}
In this section, we first set up the technical tools needed for proving
Theorem~\ref{thm:second}: a Poincar\'e inequality on $M$ with respect to the
heat measure $\dmu_{x_0}(t)$. The long-time behavior of the heat solution,
whose initial value is the pull-back measure density, is then studied following
a similar strategy as in last section.

\subsection{\textbf{Poincar\'e inequality weighted by the heat measure}} We
will control the weighted (by the heat measure) difference between a function
$f\in L^{\infty}(M)\cap C^1(M)$ and its heat evolution, roughly by the heat
evolution of its derivative (the ``central part'') and an error term (the
``tail part''). 

The estimate of the ``central part'' is based on the heat
kernel estimate of Li-Yau (Proposition~\ref{prop:kernel_estimate}), as well as
Cheeger-Colding's segment inequality (Proposition~\ref{prop:segment}). We begin
with an elementary estimate of the distance to points on a minimal geodesic by
that to the end points:
Let $x,y\in M$ and $\gamma$ be a minimal geodesic connecting them. Since
\begin{align*}
\forall s\in [0,d(x,y)],\quad d(x,\gamma(s))+d(y,\gamma(s))=d(x,y),
\end{align*}
we have
\begin{align*}
2d(x_0,\gamma(s))\ 
\le\quad &d(x_0,x)+d(x,\gamma(s))+d(x_0,y)+d(y,\gamma(s))\\
=\quad &d(x_0,x)+d(y_0,y)+d(x,y)\\
\le\quad &2d(x_0,x)+2d(y_0,y),
\end{align*}
and thus 
\begin{align*}
\frac{1}{2}r^2(\gamma(s))\le r^2(x)+r^2(y).
\end{align*}

This gives, for almost every pair of $x,y\in M$, and $s\in [0,d(x,y)]$, that
\begin{align*}
|f(x)-f(y)|\ e^{-\frac{2}{9t}\left(r^2(x)+r^2(y)\right)}\quad 
\le\quad &\int_0^{d(x,y)}|\nabla
f|(\gamma(s))e^{-\frac{2}{9t}\left(r^2(x)+r^2(y)\right)}\ \text{d}s\\
\le\quad &\int_0^{d(x,y)}|\nabla f|(\gamma(s)) e^{-\frac{1}{9t}r^2(\gamma(s))}\
\text{d}s.
\end{align*}
Using this estimate, together with Proposition~\ref{prop:kernel_estimate} and
Proposition~\ref{prop:segment}, we could control the ``central part'' as
following:
\begin{align}
\label{eqn:central_part}
\begin{split}
&\int_{B(x_0,R)}\int_{B(x_0,R)}|f(x)-f(y)|\ \dmu_{x_0}(x,t)\dmu_{x_0}(y,t)\\
\le\quad &\frac{C_2(1\slash
2)^2}{|B(x_0,\sqrt{t})|^2}\int_{B(x_0,R)}\int_{B(x_0,R)}
|f(x)-f(y)|e^{-\frac{2}{9t}\left(r^2(x)+r^2(y)\right)}\ \dvol_g(x)\dvol_g(y)\\
\le\quad &\frac{C_2(1\slash
2)^2}{|B(x_0,\sqrt{t})|^2}\int_{B(x_0,R)}\int_{B(x_0,R)}
\left(\int_0^{d(x,y)}|\nabla f|(\gamma(s)) e^{-\frac{1}{9t}r^2(\gamma(s))}\
\text{d}s\right) \dvol_g(x)\dvol_g(y)\\
\le\quad &C_{CC}(m)C_2(1\slash
2)^2\frac{|B(x_0,2R)|}{|B(x_0,\sqrt{t})|^2}(2R)\int_{B(x_0,2R)} |\nabla
f|(x)e^{-\frac{1}{9t}r^2(x)} \dvol_g(x)\\
\le\quad &C_{CC}(m)C_2(1\slash
2)^2C_1(1)3^{\frac{m}{2}}\frac{|B(x_0,2R)|}{|B(x_0,\sqrt{t})|}(2R)
\int_{B(x_0,2R)}|\nabla f|\ \dmu_{x_0}(3t).
\end{split}
\end{align}

On the other hand, we could estimate the ``tail part'' of the heat measure for
$R\ge \sqrt{t}$. Using the heat kernel upper bound and integrating radially, we
have
\begin{align*}
\int_{M\backslash B(x_0,R)}1\ \dmu_{x_0}(t)\quad 
\le\quad &\frac{C_2(1\slash
2)}{|B(x_0,\sqrt{t})|} \int_{M\backslash
B(x_0,\sqrt{t})}e^{-\frac{2r^2(x)}{9t}}\ \dvol_g(x)\\
=\quad &\frac{C_2(1\slash
2)}{|B(x_0,\sqrt{t})|} \int_R^{\infty}e^{-\frac{2r^2}{9t}} |\partial B(x_0,r)|\
\text{d}r.
\end{align*}
Since $M$ is complete and $\Rc_g\ge 0$, by the Bishop-Gromov volume comparison,
we have
\begin{align*}
|\partial B(x_0,r)|\le mr^{-1}|B(x_0,r)|\le mr^{m-1}t^{-\frac{m}{2}}|B(x_0,\sqrt{t})|,
\end{align*} 
and thus 
\begin{align}
\label{eqn:tail_part}
\begin{split}
\int_{M\backslash B(x_0,R)}1\ \dmu_{x_0}(t)\quad 
\le\quad &C_2(1\slash 2)m\int_R^{\infty}
e^{-\frac{2r^2}{9t}}r^{m-1}t^{-\frac{m}{2}}\ \text{d}r\\
=\quad &C_2(1\slash 2)m\int_{\frac{R}{\sqrt{t}}}^{\infty}
e^{-\frac{2s^2}{9}}s^{m-1}\ \text{d}s.
\end{split}
\end{align}
Denoting 
\begin{align}
\Psi_2(\sqrt{t}\slash R\ |\ m):=C_2(1\slash
2)m\int_{\frac{R}{\sqrt{t}}}^{\infty}e^{-\frac{2s^2}{9}}s^{m-1}\ \text{d}s,
\label{eqn:tail}
\end{align}
 we clearly see that 
\begin{align}
\lim_{\frac{R}{\sqrt{t}}\to \infty}\Psi_2(\sqrt{t}\slash R\ |\ m)\ =\ 0.
\label{eqn:tail_vanish}
\end{align} 

Combining the above estimates, we have the following Poincar\'e inequality:
\begin{lemma}[\textbf{Weighted Poincar\'e inequality}]
Let $f$ be a bounded function on $M$. Suppose $|f|\le L$ and $\int_M|\nabla f|\
\dmu_{x_0}(t)$ is defined for $t>0$. Then for $t>0$ and $R\ge \sqrt{t}$ we have
\begin{align*}
\int_M\left|f-\int_Mf\ \dmu_{x_0}(t)\right|\ \dmu_{x_0}(t)\ \le\
C_{WP}\frac{|B(x_0,R)|}{|B(x_0,\sqrt{t})|} R\int_M|\nabla f|\
\dmu_{x_0}(3t)+\Psi_{WP}(\sqrt{t}\slash R\ |\ m)\ L,
\end{align*}
with $C_{WP}:=C_{CC}(m)C_2(1\slash 2)^2C_1(1)6^m$ being a dimensional constant
and $\Psi_{WP}:=6\Psi_2$.
\label{lem:Weight_Poincare}
\end{lemma}

\begin{proof}
Since $\int_M1\ \dmu_{x_0}(t)=1$, we see
\begin{align*}
&\int_M\left|f-\int_Mf\ \dmu_{x_0}(t)\right|\ \dmu_{x_0}(t)\\
\le\quad &\int_{B(x_0,R)}\int_{B(x_0,R)}|f(x)-f(y)|\
\dmu_{x_0}(x,t)\dmu_{x_0}(y,t) +3\int_{M\backslash B(x_0,R)} 2L\ \dmu_{x_0}(t),
\end{align*}
then the conclusion follows easily from the previous estimates
(\ref{eqn:central_part}), (\ref{eqn:tail_part}) and (\ref{eqn:tail}) .
\end{proof}
\subsection{Large-time evolution of the pull-back measure density by the heat equation}
In this sub-section we prove Theorem~\ref{thm:second}. 
The proof follows in the same way as that of Theorem~\ref{thm:main}. However,
due to the complication in the weighted Poincar\'e inequality
(Lemma~\ref{lem:Weight_Poincare}), for any sequence $t_i\to \infty$, we need to
make a careful selection of the subsequences.

We will also need Li-Yau's Harnack inequality for the heat kernel
(Proposition~\ref{prop:Harnack}), which especially implies the following
estimate: 
If $H_{x_0}(x,t)$ is the fundamental solution to the heat equation with initial
value being the Delta function at $x_0\in M$, then for
any $t_2>t_1>0$,
\begin{align}
H_{x_0}(x,t_1)\ \le\ \left(\frac{t_2}{t_1}\right)^{\frac{3m}{4}}
H_{x_0}(x,t_2).
\label{eqn:Harnack}
\end{align}

\begin{proof}[Proof of Theorem~\ref{thm:second}]
In view of the proof of Theorem~\ref{thm:main}, we only need to show that for
any $t_i\to \infty$, there is a subsequence $t_{i_j}$ such that
\begin{align*}
\forall x\in M,\quad  \lim_{j\to \infty}\int_M|\omega|^2\ \dmu_{x_0}(t_{i_j})\
\ge\ |\omega|^2(x).
\end{align*}
In a similar way as proving Theorem~\ref{thm:main}, we pick $t_i\to \infty$ and
consider special orthogonal matrices $\bar{A}_i$ that diagonalize
$\bar{\Omega}_{t_i}(\mathbf{u}):= \left[\int_M
E_{\alpha\beta}(\mathbf{u})\ \dmu_{x_0}(t_i)\right]$. We then have some limiting
$\bar{A}_{\infty}\in SO(n)$ to which a subsequence of $\bar{A}_i$ converges.
For $\bar{\mathbf{v}}:=\bar{A}_{\infty}\mathbf{u}$, by the $SO(n)$-invariance we have
\begin{align*}
&\det \bar{\Omega}_{t_i}(\mathbf{u})\ =\ \det
\bar{\Omega}_{t_i}(\bar{A}_i\mathbf{u})\ =\ \det
\bar{\Omega}_{t_i}(\bar{\mathbf{v}}),\\
\text{and}\quad &|\omega|\ \equiv\
|\omega_{\bar{A}_i}|\ \equiv\ |\bar{A}_{\infty}\omega|\quad \text{on}\ M.
\end{align*}
Moreover, invoking Lemma~\ref{lem:heat_max}, we have
\begin{align*}
&\lim_{i\to \infty}\int_M|\nabla \bar{v}_{\alpha}|^2\ \dmu_{x_0}(t_i)\ =\
\sup_M|\nabla \bar{v}_{\alpha}|^2\ =:\ \bar{L}_{\alpha}^2,\\
\text{and}\quad
&\lim_{i\to \infty}\det \bar{\Omega}_{t_i}(\mathbf{u})\ =\
\prod_{\alpha=1}^n\bar{L}_{\alpha}^2.
\end{align*}

As made clear in the proof of Theorem~\ref{thm:main}, the key step is to obtain
a heat measure version of (\ref{eqn:det_Omega_i}). This will be our focus of
the rest of the proof.

By (\ref{eqn:splitting_error_intro}) and Lemma~\ref{lem:heat_max}, we have for
any $t>0$ and $\alpha=1,\cdots, n$,
\begin{align*}
2\int_{3t}^{4t}\int_M|\nabla\nabla \bar{v}_{\alpha}|^2\
\dmu_{x_0}(s)\text{d}s\quad 
\le\quad &\int_M|\nabla \bar{v}_{\alpha}|^2\ \dmu_{x_0}(4t)-\int_M|\nabla
\bar{v}_{\alpha}|^2\ \dmu_{x_0}(3t)\\
\le\quad &\Psi_3(t^{-\frac{1}{2}}|\ m)\ \bar{L}_{\alpha}^2,
\end{align*}
with $\Psi_3(t^{-\frac{1}{2}}|\ m)\to 0$ as $t\to \infty$. So for some
$\bar{t}\in [3t,4t]$, we have
\begin{align*}
2t\int_M|\nabla\nabla \bar{v}_{\alpha}|^2\ \dmu_{x_0}(\bar{t})\ \le\
\Psi_3(t^{-\frac{1}{2}}|\ m)\ \bar{L}_{\alpha}^2,
\end{align*}
and by Li-Yau's Harnack inequality (\ref{eqn:Harnack}), we have
\begin{align*}
2t\int_M|\nabla\nabla v_{\alpha}|^2\ \dmu_{x_0}(3t)\ \le\
\left(\frac{4}{3}\right)^{\frac{3m}{4}}\Psi_3(t^{-\frac{1}{2}}|\ m)\ \bar{L}_{\alpha}^2.
\end{align*}

For each positive integer $j$, let $i_j$ be the first index so that $\forall i\ge i_j$,
\begin{align}
2t_i\int_M|\nabla \nabla v_{\alpha}|^2\ \dmu_{x_0}(3t_i)\ \le\
j^{-2m-3}\bar{L}_{\alpha}^2.
\label{eqn:weighted_hessian_control}
\end{align}

We could now estimate as before:
\begin{align*}
&\int_M\left||\omega_{\bar{A}_{\infty}}|^2
-\det\bar{\Omega}_{t_{i_j}}(\mathbf{v})\right|\ \dmu_{x_0}(t_{i_j})\\
\le\quad &\sum_{\sigma\in S_k}\sum_{\alpha=1}^n
\left(\prod_{\beta\not=\alpha}\bar{L}_{\beta}\bar{L}_{\sigma(\beta)}\right)\int_M
\left|\langle \nabla \bar{v}_{\alpha},\nabla \bar{v}_{\sigma(\alpha)}\rangle
-\int_M\langle \nabla \bar{v}_{\alpha},\nabla \bar{v}_{\sigma(\alpha)}\rangle\ 
\dmu_{x_0}(t_{i_j})\right|\ \dmu_{x_0}(t_{i_j}).
\end{align*}
For each $\alpha=1,\cdots,n$ and $\sigma\in S_n$, we now apply the weighted
Poincar\'e inequality (Lemma~\ref{lem:Weight_Poincare}) with the choice of
$R_j:=j\sqrt{t_{i_j}}$, the H\"older inequality and
(\ref{eqn:weighted_hessian_control}) to see:
\begin{align*}
&\int_M\left |\langle \nabla \bar{v}_{\alpha},\nabla
\bar{v}_{\sigma(\alpha)}\rangle-\int_M \langle \nabla \bar{v}_{\alpha},\nabla
\bar{v}_{\sigma(\alpha)}\rangle\ \dmu_{x_0}(t_i)\right|\ \dmu_{x_0}(t_i)\\
\le\quad &C_{WP}j^{m+1}\sqrt{t_{i_j}}\int_M|\nabla\langle \nabla
v_{\alpha},\nabla v_{\sigma(\alpha)}\rangle|\ \dmu_{x_0}(3t_{i_j})
+\Psi_{WP}(j^{-1} |\ m)\ \bar{L}_{\alpha}\bar{L}_{\sigma(\alpha)}\\
\le\quad&2C_{WP}j^{-\frac{1}{2}}\bar{L}_{\alpha}\bar{L}_{\sigma(\alpha)}
+\Psi_{WP}(j^{-1} |\ m)\ \bar{L}_{\alpha}\bar{L}_{\sigma(\alpha)}.
\end{align*}
Therefore as $j\to \infty$,
\begin{align*}
\int_M\left||\omega_{\bar{A}_{\infty}}|^2-\det
\bar{\Omega}_{t_{i_j}}(\mathbf{v})\right|\ \dmu_{x_0}(t_{i_j})\ 
\le\ &\left(C_{WP}j^{-\frac{1}{2}}+\Psi_{WP}(j^{-1}|\ m)\right)\
\prod_{\alpha=1}^n\bar{L}_{\alpha}^2\ \to\ 0.
\end{align*}
This is the heat measure analogue of (\ref{eqn:det_Omega_i}) and it suffices to
conclude the proof as argued before.
\end{proof}


\section{Discussion}



\subsection{Sub-harmonicity of the pull-back energy density} 
Continuing with the discussion and notations in the introduction, we recall that
we have found $|\omega|^2$, as a bounded, non-negative function on $M$, but is
\emph{not} necessarily sub-harmonic. However, there has not been any
non-trivial example where the associated pull-back density of a vector valued
harmonic map is shown not to be sub-harmonic. Notice that any meaningful
example should be considered on a manifold which is non-parabolic (say, the
volume of radius $r$ geodesic ball grows at least $\approx r^3$) and with
vanishing volume ratio at infinity.

If it were shown that $\Delta |\omega|^2\ge 0$, then we would not need
Theorem~\ref{thm:second} to conclude (\ref{eqn:identity}). However, the
following computation indicates that one should hardly expect $|\omega|^2$ to be
sub-harmonic: Recall that $E_{\alpha\beta}=\langle
\nabla u_{\alpha},\nabla u_{\beta}\rangle$ and $\E=[E_{\alpha\beta}]$, a
non-negative definite matrix valued function on $M$. Then $|\omega|^2=\det \E$,
and
\begin{align*}
\Delta \det\E\quad
=\quad &\sum_{\alpha, \beta=1}^n(-1)^{\alpha+\beta}\ \nabla \cdot \left(\det
\E^{\ast}_{\alpha;\beta}\ \nabla E_{\alpha\beta} \right)\\
=\quad
&\sum_{\alpha, \beta=1}^n(-1)^{\alpha+\beta}
\left(\det \E^{\ast}_{\alpha;\beta}\ \Delta E_{\alpha\beta}+
\sum_{\gamma\not=\alpha,\delta\not=\beta}(-1)^{\gamma+\delta}\det
\E^{\ast}_{\alpha\gamma;\beta\delta}\ \langle \nabla
E_{\gamma\delta}, \nabla E_{\alpha\beta}\rangle \right)\\
=\quad &\sum_{\alpha,
\beta=1}^n \sum_{\gamma\not=\alpha,\delta\not=\beta} (-1)^{\alpha+\beta+\gamma+\delta}\
\det \E_{\alpha\gamma;\beta\delta}^{\ast}\left(E_{\gamma\delta}\ 
\Delta E_{\alpha\beta} +\langle \nabla E_{\gamma\delta},\nabla
E_{\alpha\beta}\rangle  \right),
\end{align*}
where we use $\E^{\ast}_{\alpha_1\cdots \alpha_k;\beta_1\cdots\beta_k}$ to
denote the $(n-k,n-k)$-matrix obtained by deleting row
$\alpha_1,\dots,\alpha_k$ and column $\beta_1,\dots,\beta_k$ from $\E$.

\subsection{Directions of future research}
The results in this paper leave the following directions open for further
investigation:
\begin{enumerate}
\item For a vector valued harmonic map $\mathbf{u}:M^m\to \R^n$ ($2<n\le m$),
we consider the matrix $\E$ and denote the
$k^{\text{th}}$-symmetric polynomial of the eigenvalues of $\E$ by $\sigma_k$
($k=1,\cdots, n$). Suppose $(M^m,g)$ is complete, non-compact with non-negative
Ricci curvature, and $\mathbf{u}$ is of at most linear growth, then Li's
identities say that
\begin{align*}
\quad\quad\quad\quad \lim_{\rho\to \infty}\fint_{B(x_0,\rho)}\sigma_1\ \dvol_g\
=\ \sup_M\sigma_1\ =\ \lim_{t\to \infty}\int_M\sigma_1\ \dmu_{x_0}(t);
\end{align*}
while Theorem~\ref{thm:main} and Theorem~\ref{thm:second} together tell that
\begin{align*}
\quad \quad \quad\quad \lim_{\rho\to \infty}\fint_{B(x_0,\rho)}\sigma_n\
\dvol_g\ =\ \sup_M\sigma_n\ =\ \lim_{t\to \infty}\int_M\sigma_n\ \dmu_{x_0}(t).
\end{align*}
It is therefore interesting to ask whether such identities hold for the
symmetric polynomials in between, i.e. for $\sigma_{k}$ with $2\le k\le n-1$,
and whether achieving a global maximum of any of such $\sigma_{k}$'s will
induce isometric splitting of $M$.

\item On the other direction, it is interesting to consider more general
harmonic maps $\mathbf{u}:(M^m,g)\to (N^n,h)$ into a Cartan-Hardamad manifold
$N$ (i.e. $(N,h)$ has non-positive sectional curvature everywhere). Assume that
$(M^m,g)$ is complete, non-compact and has non-negative Ricci curvature, then
by the Weitzenb\"ock formula, we have
\begin{align*}
\quad \quad \quad\quad \Delta|\nabla \mathbf{u}|^2\quad
 =\quad &2|\nabla\nabla
\mathbf{u}|^2+2\sum_{\alpha,\beta=1}^m
Tr_h\left(\Rc^M_{\alpha\beta}\nabla_{\alpha}\mathbf{u}\nabla_{\beta}\mathbf{u}\right)
-\sum_{\alpha,\beta=1}^m\left(\mathbf{u}^{\ast}\Rm^N\right)_{\alpha\beta\beta\alpha}\\
\ge\quad &2|\nabla\nabla \mathbf{u}|^2.
\end{align*}
It is clear that if $\mathbf{u}$ has uniformly bounded energy density, then
same argument as before (see Lemma~\ref{lem:heat_max}) gives
\begin{align*}
\quad\quad\quad\quad \lim_{t\to \infty}\int_M|\nabla \mathbf{u}|^2\
\dmu_{x_0}(t)\ =\ \sup_M|\nabla \mathbf{u}|^2.
\end{align*}
Thus if $|\nabla \mathbf{u}|^2$ achieves the global maximum, then
$|\nabla\nabla \mathbf{u}|\equiv 0$, so $N$ has to be the $n$-dimensional
Euclidean space and $M$ isometrically splits $\R^n$. Again, it is interesting
to see if such maximum principle still holds for the corresponding $\sigma_{k}$
for $k>1$, i.e. if any $\sigma_{k}$ sees a global maximum on $M$, then both
$N\equiv \R^n$ and $M\equiv M'\times \R^n$. See~\cite{LiWang} for more
information in this direction.

\end{enumerate}

\quad

\quad

\end{document}